\documentclass[12pt,reqno]{amsart}
\usepackage{amssymb}

\def\jb#1{\langle#1\rangle}
\def\norm#1{\|#1\|}
\def\normo#1{\left\|#1\right\|}

\newcommand{\F}{\mathcal{F}}

\newcommand{\cP}{\mathcal{P}}

\newcommand{\C}{\mathbb{C}}

\newcommand{\R}{\mathbb{R}}
\newcommand{\Z}{\mathbb{Z}}

\newcommand{\al}{\alpha}

\newcommand{\e}{\varepsilon}
\newcommand{\fy}{\varphi}
\newcommand{\om}{\omega}

\newcommand{\x}{\xi}
\newcommand{\y}{\eta}

\newcommand{\ro}{\rho}
\newcommand{\ft}{{\mathcal{F}}}

\newcommand{\De}{\Delta}
\newcommand{\Om}{\Omega}

\newcommand{\p}{\partial}
\newcommand{\na}{\nabla}
\newcommand{\re}{\mathop{\mathrm{Re}}}

\newcommand{\lec}{\lesssim}

\newcommand{\I}{\infty}

\newcommand{\ti}{\widetilde}

\newcommand{\LR}[1]{{\langle #1 \rangle}}

\newcommand{\EQ}[1]{\begin{equation}\begin{split} #1 \end{split}\end{equation}}
\newcommand{\Del}[1]{}
\newcommand{\CAS}[1]{\begin{cases} #1 \end{cases}}

\newcommand{\pt}{&}
\newcommand{\pr}{\\ &}
\newcommand{\pq}{\quad}
\newcommand{\pn}{}

\numberwithin{equation}{section}

\newtheorem{thm}{Theorem}[section]

\newtheorem{lem}[thm]{Lemma}

\theoremstyle{remark}
\newtheorem{rem}{Remark}

\begin{document}

\title[small energy scattering for Zakharov]{Small energy scattering for the Zakharov system with radial symmetry}
\author[Z. Guo, K. Nakanishi]{Zihua Guo$^{1,2}$, Kenji Nakanishi$^3$}

\address{$^1$LMAM, School of Mathematical Sciences, Peking
University, Beijing 100871, China}
\address{$^2$Beijing International Center for Mathematical Research, Beijing
100871, China}
\email{zihuaguo@math.pku.edu.cn}

\address{$^3$Department of Mathematics, Kyoto University, Kyoto 606-8502,
Japan}

\email{n-kenji@math.kyoto-u.ac.jp}

\begin{abstract}
We prove small energy scattering for the 3D Zakharov system with
radial symmetry. The main ingredients are normal form reduction and
the radial-improved Strichartz estimates.
\end{abstract}

\maketitle

\section{Introduction}
In this paper, we consider the Cauchy problem for the 3D Zakharov
system \EQ{\label{eq:Zak}
 \CAS{ i\dot u - \De u = nu,\\
   \ddot n/\al^2 - \De n = -\De|u|^2,}}
with the initial data
\begin{align}
u(0,x)=u_0,\, n(0,x)=n_0,\,\dot n(0,x)=n_1,
\end{align}
where $(u,n)(t,x):\R^{1+3}\to\C\times\R$, and $\al>0$ denotes the
ion sound speed. It preserves $\|u(t)\|_{L^2_x}$ and the energy \EQ{
 E=\int_{\R^3}|\na u|^2+\frac{|D^{-1}\dot n|^2/\al^2+|n|^2}{2}-n|u|^2 dx,}
where $D:=\sqrt{-\De}$, as well as the radial symmetry
\EQ{
 (u,n)(t,x)=(u,n)(t,|x|).}
We consider those solutions with such symmetry and finite energy,
hence \EQ{\label{eq:indata}
 (u_0,n_0,n_1)\in H^1_r(\R^3)\times L^2_r(\R^3)\times \dot H^{-1}_r(\R^3).}
We are interested in the scattering for small data in the above
function space, as it will open the possibility to discuss about
global dynamics of general large solutions (under the radial
symmetry).

This system \eqref{eq:Zak} was introduced by Zakharov \cite{Zak} as
a mathematical model for the Langmuir turbulence in unmagnetized
ionized plasma. It has been extensively studied. Local wellposedness
(without symmetry) is well known. For example, the well-posedness in
the energy space was proved in \cite{BoCo} for $d = 2, 3$ and in
\cite{GTV} for $d=1$. 
The improvement was obtained to the critical regularity in
\cite{GTV,BHHT} for $d=1,2$, and to the full subcritical regularity
in \cite{GTV,BeHe} for $d\geq 4,d=3$. The well-posedness for the
system on the torus was studied in \cite{Takaoka,Kishi}. These
results all follow from the iteration argument by using Bourgain's $X^{s,b}$ space. For
the subsonic limit to NLS (as $\alpha\to \infty$), uniform local wellposedness and convergence to a NLS solution has been obtained in 
\cite{SW,OT,KPV,MN}. Concerning the long time and blow-up behavior,
Merle \cite{Merle} obtained blow-up in finite or infinite time for
negative energy, and the scattering theory was studied in 
\cite{OT2,Shimo,GV} for prescribed asymptotic states (namely the final data problem) in weighted Sobolev spaces. 

The main difficulties for the scattering in the Sobolev space are derivative loss and slow dispersion of the 
wave equation together with the quadratic nonlinearity. Our idea is
to combine the normal form technique, which was first used in a
dispersive PDE context by Shatah \cite{Shatah}, and the improved
Strichartz estimate for radial functions that was obtained recently
in \cite{GuoWang}. The normal form transform was used in \cite{OTT} for the Klein-Gordon-Zakharov system (that is the system where the Schr\"odinger equation is replaced with the Klein-Gordon) and they got the scattering from initial data small in weighted Sobolev spaces without symmetry. Their estimates rely on Klainerman's vector field method, which is not directly applicable to the standard Zakharov system \eqref{eq:Zak} because of the mixed linear part. 

\begin{thm}\label{thm}
If $(u_0,n_0,n_1)$ are all radial and is small enough in the norm of
\eqref{eq:indata}, then the solution $(u,n)$ scatters in this space
as $t\to\pm\I$. More precisely, there are unique solutions $(u^\pm,n^\pm)$ of $0=i\dot u^\pm-\De u^\pm=\ddot n^\pm/\al^2-\De n^\pm$ satisfying 
\EQ{
 \lim_{t\to\pm\I}\left[\|u(t)-u^\pm(t)\|_{H^1} + \|n(t)-n^\pm(t)\|_{L^2}+\|\dot n(t)-\dot n^\pm(t)\|_{\dot H^{-1}}\right] = 0.}
\end{thm}

\begin{rem}
Global well-posedness (without symmetry) with small norm in the energy
space was proved in \cite{BoCo}. The novelty of Theorem \ref{thm} is
that the solution with radial symmetry also scatters. 
Our proof is much simpler than those analysis \cite{OT2,Shimo,GV,OTT} in weighted Sobolev spaces.  
\end{rem}

\section{Transform of equation}

This section is devoted to transform the equation by using the
normal form.  It is convenient first to change the system into first
order as usual. Let \EQ{
 N:=n - iD^{-1}\dot n/\al,}
then $n=\re N=(N+\bar N)/2$ and the equations for $(u,N)$ are
\EQ{\label{eq:Zak1}
 \CAS{ (i\p_t-\De) u =Nu/2+\bar N u/2,\\
   (i\p_t+\al D) N = \al D|u|^2.}}
One can easily see that 
\EQ{
 \|n\|_{L^2}^2 + \|D^{-1}\dot n\|_{L^2}^2/\al^2 = \|N\|_{L^2}^2,}
and a free wave $n(t)$ is transformed into a free wave $N(t)=e^{it\al D}N(0)$. 
In our proof, the term $\bar N u$ makes no essential difference from 
$Nu$, and hence for simplicity, we assume the nonlinear term in
first equation of \eqref{eq:Zak1} is $Nu$.

Now we introduce some notations. We use $S(t),W_\alpha(t)$ to denote
the Schr\"odinger, wave semigroups:
\[S(t)\phi=\ft^{-1}e^{it|\xi|^2}\hat{\phi},\quad W_\alpha(t)\phi=\ft^{-1}e^{i\alpha t|\xi|}\hat{\phi}, \pq \hat\phi=\F\phi.\]
Let $\eta_0: \R\rightarrow [0, 1]$ denote an radial smooth function
supported in $\{|\xi|\leq 8/5\}$ and equal to $1$ in $\{|\xi|\leq
5/4\}$. For $k\in \Z$ let
$\chi_k(\xi)=\eta_0(\xi/2^k)-\eta_0(\xi/2^{k-1})$ and $\chi_{\leq
k}(\xi)=\eta_0(\xi/2^k)$. For $k\in \Z$ let $P_k, P_{\leq k}$ denote
the operators on $L^2(\R^3)$ defined by
$\widehat{P_ku}(\xi)=\chi_k(|\xi|)\widehat{u}(\xi),\widehat{P_{\leq
k}u}(\xi)=\chi_{\leq k}(|\xi|)\widehat{u}(\xi)$.

For any pair of functions $u$ and $v$, we use $(uv)_{LH}$, $(uv)_{HH}$,
$(uv)_{HL}$ to denote the three different interactions
\[
 (uv)_{LH}:=\sum_{k\in\Z}P_{\leq k-5}uP_kv,(uv)_{HL}:=\sum_{k\in\Z}P_kuP_{\leq k-5}v,(uv)_{HH}:=\sum_{\substack{|k_1-k_2|\leq 4 \\ k_1,k_2\in\Z}}P_{k_1}uP_{k_2}v.\]  
To distinguish the resonant interaction, we also use 
\EQ{
 (uv)_{\al L}=\sum_{\substack{|k-\log_2\al|\le 1,\\ k\in\Z}}P_kuP_{\leq k-5}v,
 (uv)_{X L}=\sum_{\substack{|k-\log_2\al|> 1,\\ k\in\Z}}P_kuP_{\leq k-5}v,}
and similarly $(uv)_{L\al}$, $(uv)_{LX}$. 
It is obvious that we have 
\EQ{
 uv\pt=(uv)_{HH}+(uv)_{LH}+(uv)_{HL}
 \pr=(uv)_{HH}+(uv)_{L\al}+(uv)_{LX}+(uv)_{\al L}+(uv)_{XL},}
and they are all radial if $u,v$ are both radial. 
All these expressions should be understood as bilinear operators acting on $u$ and $v$, rather than operators on the product $uv$. 
More explicitly, for any such index $*=HH,HL,\al L$, etc., we denote the bilinear symbol (multiplier) by 
\EQ{
 \F(uv)_* = \int \cP_*\hat u(\x-\y)\hat v(\y)d\y.}
Finite sum of those bilinear operators are denoted by the sum of indices:
\EQ{
 (uv)_{*_1+*_2+\cdots}=(uv)_{*_1}+(uv)_{*_2}+\cdots.}

From Duhamel's formula and taking a Fourier transform, we get that
the first equation of \eqref{eq:Zak1} is equivalent to
\begin{align*}
\hat{u} &=e^{it|\xi|^2}\hat{u_0}-i\int_0^t
e^{i(t-s)|\xi|^2}\ft(Nu)_{XL} ds
 -i\int_0^t e^{i(t-s)|\xi|^2}\ft(Nu)_{HH+LH+\al L}ds\\
&:=I+II+III.
\end{align*}
Using the equation
\eqref{eq:Zak1} again, we get that
\begin{align}
\partial_t(e^{-it|\xi|^2}\hat{u})=&-ie^{-it|\xi|^2}(\hat{N}*\hat{u})(\xi),\\
\partial_t(e^{-i\alpha t|\xi|}\hat{N})=&-ie^{-i\alpha t|\xi|}\alpha |\xi|(\hat{u}*\hat{\bar
u})(\xi).
\end{align}
Thus we have
\begin{align*}
II=&-i\int_0^t e^{i(t-s)|\xi|^2}\cP_{XL} \hat{N}(s,\xi-\eta)\hat{u}(s,\eta)d\eta
ds\\
=&-ie^{it|\xi|^2}\int_0^t \int \cP_{XL} e^{is\om}[e^{-i\alpha s|\xi-\eta|}\hat{N}(s,\xi-\eta)][e^{-is|\eta|^2}\hat{u}(s,\eta)]d\eta ds,
\end{align*}
where the resonance function \EQ{
 \om:=-|\x|^2+\al|\x-\y|+|\y|^2}
in the support of $\cP_{XL}$: $\alpha \nsim |\xi-\eta|\sim |\xi|\gg |\eta|$, has the
following size \EQ{
 |\om| \sim |\x|^2+\al|\x|.}
This means roughly that for the non-resonant interactions in $\cP_{XL}$, we gain $|\x|^{-2}$ for high frequencies ($|\x|>1$), and lose $|\x|^{-1}$ for low frequencies ($|\x|<1$), by the time integration. The gain for high frequencies has been used for the local wellposedness in the $X^{s,b}$ spaces. In general, the lower frequencies can be more problematic in the scattering problems, but it will turn out that we can absorb $|\x|^{-1}$ by the Sobolev embedding.

Thus from integration by parts, we get
\begin{align*}
II=&-e^{it|\xi|^2}\int_0^t \int \cP_{XL}\om^{-1}\partial_s(e^{is\om })e^{-i\alpha
s|\xi-\eta|}\hat{N}(s,\xi-\eta)e^{-is|\eta|^2}\hat{u}(s,\eta)d\eta ds\\
=&-\int\cP_{XL}\om^{-1}[\hat{N}(t,\xi-\eta)\hat{u}(t,\eta)-e^{it|\xi|^2}\hat{N}(0,\xi-\eta)\hat{u}(0,\eta)]d\eta\\
&-\alpha\int_0^t \int \cP_{XL}\om^{-1}e^{i(t-s)|\xi|^2
}|\xi-\eta|\widehat{|u|^2}(\xi-\eta)\hat{u}(s,\eta)d\eta
ds\\
&-\frac{1}{2}\int_0^t \int e^{i(t-s)|\xi|^2 }\cP_{XL}\om^{-1}\hat{N}(s,\xi-\eta)(\hat{N}*\hat{u})(\eta)d\eta ds.
\end{align*}
We introduce a bilinear Fourier multiplier in the form 
\EQ{
 \Om(f,g)=\F^{-1}\int\cP_{XL}\om^{-1}\hat f(\x-\y)\hat g(\y)d\y.} 
Then we have
\begin{align*}
II=&-\ft\Om(N,u)(t)+e^{it|\xi|^2}\ft\Om(N,u)(0)\\
&-i\alpha\int_0^te^{i(t-s)|\xi|^2}\ft
\Om(D|u|^2,u)(s)ds-\frac{i}{2}\int_0^te^{i(t-s)|\xi|^2}\ft
\Om(N,Nu)(s)ds.
\end{align*}
Thus we obtain
\begin{align}\label{eq:intu}
u=&S(t)u_0+S(t)\Om(N,u)(0)-\Om(N,u)(t)-i\alpha\int_0^tS(t-s)\Om(D|u|^2,u)(s)ds\nonumber\\
&-\frac{i}{2}\int_0^tS(t-s)
\Om(N,Nu)(s)ds-i\int_0^tS(t-s)(Nu)_{LH+HH+\alpha L}ds.
\end{align}

For the second equation in \eqref{eq:Zak1}, similarly, we can apply
the normal form reduction for the high-low interaction, and then get
that it is equivalent to
\begin{align}\label{eq:intN}
N=&W_\alpha(t)N_0+W_\alpha(t)D\tilde\Om(u,u)(0)-D\tilde\Om(u,u)(t)-\int_0^tW_\alpha(t-s)D(u\bar
u)_{HH+\alpha L+L\al}ds\nonumber\\
&-\int_0^tW_\alpha(t-s)(D\tilde\Om(Nu,u)+D\tilde\Om(u,Nu))(s)ds,
\end{align}
where $\tilde \Om$ is a bilinear Fourier multiplier in the form \EQ{
 \tilde\Om(f,g)=\F^{-1}\int \cP_{XL+LX}\frac{\hat f(\x-\y)\hat{\bar g}(\y)}{|\xi-\eta|^2-|\eta|^2-\alpha|\xi|}d\y.}
It's easy to see that $\tilde \Om$ behaves "essentially" the same as
$\Om$. The transformed equations look "roughly" (neglecting the difference between $N$ and $\bar{N}$), 
\EQ{
 \pt(i\p_t+D^2)(u-\Om(N,u))=(Nu)_{LH+HH+\al L}+\Om(D|u|^2,u)+\Om(N,Nu),
 \pr(i\p_t+\al D)(N-D\tilde\Om(u,u))=D(u\bar{u})_{HH+\al L+L\al}+D\tilde \Om(Nu,u)+D\tilde\Om(u,Nu).}
In the normal form reduction, we didn't use the radial symmetry. We
also remark that this reduction doesn't destroy the symmetry.

\section{Strichartz estimates and nonlinear estimates}

In this section, we introduce the Strichartz norm we need. Because
of the quadratic term, our spaces rely heavily on the radial
symmetry. The homogeneous Besov and the inhomogeneous Sobolev spaces on $\R^3$ are defined respectively by
\EQ{
 \pt \|\fy\|_{\dot B^s_{p,q}}=\|2^{ks}P_k\fy\|_{\ell^q(\Z;L^p(\R^3))},
 \pq \|\fy\|_{H^s_p}=\|\LR{D}^s\fy\|_{L^p(\R^3)}.}
The $L^2$-Sobolev spaces are denoted by $\dot H^s=\dot B^s_{2,2}$, $H^s=H^s_2$. 
For $u$ and $N$, we use the radial-improved Strichartz
norms \EQ{ \label{Strz norms}
 \pt u \in \LR{D}^{-1}(L^\I_tL^2_x \cap L^2\dot B^{1/4+\e}_{q(\e),2}),
 \pr N \in L^\I_tL^2_x \cap L^2_t\dot B^{-1/4-\e}_{q(-\e),2},}
for fixed $0<\e\ll 1$, where $q(\cdot)$ is defined by \EQ{
 \frac{1}{q(\e)}=\frac{1}{4}+\frac{\e}{3},}
such that we have the Sobolev embedding
\EQ{ \label{Sobolev}
 \dot B^{s_1}_{q(s_1),\ro_1} \subset \dot B^{s_2}_{q(s_2),\ro_2} \pq (s_1\ge s_2, \ro_1\ge \ro_2).}
In particular, we have
\EQ{
 \dot H^1 = \dot B^1_{q(3/4),2} \subset \dot B^{1/4+\e}_{q(\e),2} \subset \dot B^{1/4-\e}_{q(-\e),2}\subset L^6.}
From now on, the third exponent of the Besov space will be fixed to
$2$ and so omitted. The condition $0<\e\ll 1$ ensures that \EQ{
 \frac{10}{3}<q(\e)<4<q(-\e)<\I,}
such that the norms in \eqref{Strz norms} are Strichartz-admissible
for radial solutions. The Strichartz estimates that we will use are
given in the following lemma, and refer to \cite{GuoWang} for their
proof as well as some preceding results. 

\begin{lem}\label{lem:radstri} Assume $\phi(x)$, $f(t,x)$ are spatially radially
symmetric in $\R^3$. Then

(1) if $(q,r),(\tilde q,\tilde r)$ both satisfy the
Schr\"odinger-admissible condition:
\[2\leq q \leq \infty, \frac{2}{q}+\frac{5}{r}<\frac{5}{2}; \mbox{ or } (q,r)=(\infty,2)\]
and $\tilde q>2$, then
\begin{align*}
\norm{S(t)\phi}_{L_t^q\dot{B}_{r,2}^{\frac{2}{q}+\frac{3}{r}-\frac{3}{2}}} &\lec
\norm{\phi}_{L^2},\\
\normo{\int_0^tS(t-s)f(s)ds}_{L_t^q\dot{B}_{r,2}^{\frac{2}{q}+\frac{3}{r}-\frac{3}{2}}} &\lec
\norm{f}_{L_t^{\tilde q'}\dot B_{\tilde
r',2}^{\frac{3}{2}-\frac{2}{\tilde q}-\frac{3}{\tilde r}}}.
\end{align*}

(2) if $(q,r),(\tilde q,\tilde r)$ both satisfy the wave-admissible
condition:
\[2\leq q \leq \infty, \frac{1}{q}+\frac{2}{r}<1; \mbox{ or } (q,r)=(\infty,2)\]
and $\tilde q>2$, then
\begin{align*}
\norm{W_\alpha(t)\phi}_{L_t^q\dot{B}_{r,2}^{\frac{1}{q}+\frac{3}{r}-\frac{3}{2}}}&\lec
\norm{\phi}_{L^2},\\
\normo{\int_0^tW_\alpha(t-s)f(s)ds}_{L_t^q\dot{B}_{r,2}^{\frac{1}{q}+\frac{3}{r}-\frac{3}{2}}}&\lec
\norm{f}_{L_t^{\tilde q'}\dot B_{\tilde
r',2}^{\frac{3}{2}-\frac{1}{\tilde q}-\frac{3}{\tilde r}}}.
\end{align*}
\end{lem}

We intend to apply this lemma to the integral equations. Then in
order to close the argument, we need to do some nonlinear estimates.

\subsection{Bilinear terms} The above Strichartz norms neatly fit in
the bilinear terms on the right, which are partially resonant.
Indeed we have
\begin{lem}\label{lem:bi1} (1) For any $N$ and $u$, the following estimates hold
\begin{align*}
 \|(Nu)_{LH}\|_{L^1_tH^1_x} &\lec \|N\|_{L^2_t\dot B^{-1/4-\e}_{q(-\e)}}\|\jb{D}u\|_{L^2_t \dot
 B^{1/4+\e}_{q(\e)}},\\
 \|(Nu)_{HH}\|_{L^1_tH^1_x} &\lec \|N\|_{L^2_t\dot B^{-1/4-\e}_{q(-\e)}}\|\jb{D}u\|_{L^2_t \dot
 B^{1/4+\e}_{q(\e)}}.
\end{align*}

(2) If $0\leq \theta\leq 1$, $\frac{1}{\tilde
q}=\frac{1}{2}-\frac{\theta}{2}$, $\frac{1}{\tilde
r}=\frac{1}{4}+\frac{\theta}{3}+\frac{\e}{3}$, then for any $N$ and
$u$
\begin{align*}
 \|(Nu)_{\alpha L}\|_{\LR{D}^{-1}L^{\tilde q'}_t \dot
B^{\frac{3}{2}-\frac{2}{\tilde q}-\frac{3}{\tilde r}}_{\tilde r'}}
\lec \|N\|_{L^2_t\dot B^{-1/4-\e}_{q(-\e)}}\|u\|_{L_t^\infty
L_x^2\cap L^2_t \dot B^{1/4+\e}_{q(\e)}}.
\end{align*}
\end{lem}
\begin{proof}
First we prove (1). For the first inequality, it suffices to prove
\[ \|(Nu)_{LH}\|_{H^1_x} \lec \|N\|_{\dot B^{-1/4-\e}_{q(-\e)}}\|\jb{D}u\|_{\dot B^{1/4+\e}_{q(\e)}}.\]
By dyadic decomposition, we have $(Nu)_{LH}=\sum_{k_1\leq
k_2-5}P_{k_1}NP_{k_2}u$. Then by H\"older inequality, we get
\begin{align*}
&\norm{\jb{D}(Nu)_{LH}}_{L^2}\\
&\leq (\sum_{k_2}|\sum_{k_1\leq k_2-5}(2^{k_2}+1)\norm{P_{k_1}N}_{q(-\e)}\norm{P_{k_2}u}_{q(\e)}|^2)^{1/2}\\
&\lec (\sum_{k_2}|\sum_{k_1\leq k_2-5}2^{k_1(\frac{1}{4}+\e)}2^{k_1(-\frac{1}{4}-\e)}\norm{P_{k_1}N}_{q(-\e)}(1+2^{k_2})\norm{P_{k_2}u}_{q(\e)}|^2)^{1/2}\\
&\lec\|N\|_{\dot B^{-1/4-\e}_{q(-\e)}}\|\jb{D}u\|_{\dot
B^{1/4+\e}_{q(\e)}}.
\end{align*}
Similarly, for the second one, we get
\begin{align*}
\norm{\jb{D}(Nu)_{HH}}_{L^2}&\leq \sum_{|k_1-k_2|\lec
4}\norm{P_{k_1}N}_{q(-\e)}(1+2^{k_2})\norm{P_{k_2}u}_{q(\e)}\\
&\lec\|N\|_{\dot B^{-1/4-\e}_{q(-\e)}}\|\jb{D}u\|_{\dot
B^{1/4+\e}_{q(\e)}}.
\end{align*}
For (2), as $(Nu)_{\alpha L}$ is supported in a fixed band
$|\x|\sim\al$, then we have \EQ{ \|(Nu)_{\alpha
L}\|_{\LR{D}^{-1}L^{\tilde q'}_t \dot B^{\frac{3}{2}-\frac{2}{\tilde
q}-\frac{3}{\tilde r}}_{\tilde r'}}
 \sim  \|(Nu)_{\alpha L}\|_{L^{\tilde q'}_t L^{\tilde r'}_x}
 \lec \|N_\alpha\|_{L^2_t L^{q(-\e)}_x}\|u_{<\al}\|_{L^{\frac{2}{\theta}}_t L^{\frac{6}{3-2\theta}}_x}.}
Thus (2) follows by Sobolev embedding and interpolation.
\end{proof}
\begin{rem}
In application, we will use Lemma \ref{lem:bi1} (2) by fixing a
$\theta_0$ such that $0<\theta_0<\frac{3}{8}-\frac{5\e}{2}$. By this
choice, $(\tilde q,\tilde r)$ is admissible to apply Lemma
\ref{lem:radstri} (1).
\end{rem}

\begin{lem}\label{lem:bi2} (1) For any $u$, we have
 \EQ{
 \|D(u\bar u)_{HH}\|_{L^1_tL^2_x} \lec \|u\|_{L^2_t\dot B^{1/4-\e}_{q(-\e)}}\|\jb{D}u\|_{L^2_t \dot B^{1/4+\e}_{q(\e)}}.}

(2) If $0\leq \theta\leq 1$, $\frac{1}{\tilde
q}=\frac{1}{2}-\frac{\theta}{2}$, $\frac{1}{\tilde
r}=\frac{1}{4}+\frac{\theta}{3}-\frac{\e}{3}$, then \EQ{
 \|D(u\bar u)_{\alpha L+L\al}\|_{L^{\tilde q'}_t \dot
B^{\frac{3}{2}-\frac{1}{\tilde q}-\frac{3}{\tilde r}}_{\tilde r'}}
\lec \|\jb{D}u\|_{L_t^\infty L_x^2\cap L^2_t \dot
B^{1/4+\e}_{q(\e)}}^2.}
\end{lem}
\begin{proof}The proof is similar to that of previous lemma. We omit
the details.
\end{proof}

\begin{rem}
In application, we will use Lemma \ref{lem:bi2} (2) by fixing a
$\theta_0$ such that $0<\theta_0<4\e$. By this choice, $(\tilde
q,\tilde r)$ is admissible to apply Lemma \ref{lem:radstri} (2).
\end{rem}

\subsection{Boundary terms}
Next, we estimate the boundary terms.
\begin{lem}
For any $N_0$ and $u_0$, we have
\begin{align*}
\|\Om(N_0,u_0)\|_{H^{1}_x} \lec \|N_0\|_{L^2_x}\|u_0\|_{H^1_x},\,
\|D\tilde\Om(u_0,u_0)\|_{L^2_x}\lec \|u_0\|_{H^1_x}\|u_0\|_{H^1_x}.
\end{align*}
As a consequence, for any $N$ and $u$
\begin{align*}
\|\Om(N,u)\|_{L_t^\I H^{1}_x} \lec \|N\|_{L_t^\I L^2_x}\|u\|_{L_t^\I
H^1_x},\, \|D\tilde\Om(u,u)\|_{L_t^\I L^2_x}\lec \|u\|_{L_t^\I
H^1_x}\|u\|_{L_t^\I H^1_x}.
\end{align*}
\end{lem}
\begin{proof}
We only prove $\|\Om(N_0,u_0)\|_{H^{1}_x} \lec
\|N_0\|_{L^2_x}\|u_0\|_{H^1_x}$, since the others are similar. 
By Plancherel, Young and Schwarz we have
\EQ{
\|\Om(N_0,u_0)\|_{H^{1}_x} \pt\lec \normo{\int_{|\xi-\eta|\gg|\eta|} |\eta|^{-1}|\hat{N_0}(\xi-\eta)|\cdot|\hat{u_0}(\eta)|d\eta}_{L^2_\x}
 \pr\le\|\hat N_0\|_{L^2_\x} \||\x|^{-1}\hat u_0\|_{L^1_\x}
 \pn\le \|\hat N_0\|_{L^2_\x} \||\x|^{-1}\LR{\x}^{-1}\|_{L^2_\x}\|\LR{\x}\hat u_0\|_{L^2_\x}
 \pr\lec \|N_0\|_{L^2_x}\|u_0\|_{H^1_x}.}
\end{proof}

To handle the other component, we will need the following
Coifman-Meyer type bilinear multiplier estimates:

\begin{lem}\label{lem:bimult}
Let $T_m$ be the bilinear operator on $\R^n(n\geq 1)$
\[T_m(f,g)(x)=\int_{\R^{2n}} m(\xi,\eta)\hat{f}(\xi)\hat{g}(\eta)e^{ix(\xi+\eta)}d\xi d\eta.\]
Assume $m$ is bounded and satisfies for all $\alpha,\beta$
\[|\partial_\xi^\alpha \partial_\eta^\beta m(\xi,\eta)|\leq C_{\alpha\beta}|\xi|^{-|\alpha|}|\eta|^{-|\beta|}.\]
If $1\leq p,q,r\leq \infty$ and $1/r=1/p+1/q$, then for any
$k_1,k_2\in \Z$ we have
\[\norm{T_m(P_{k_1}f,P_{k_2}g)}_{L^r}\leq C\norm{f}_{L^p}\norm{g}_{L^q}.\]
\end{lem}
\begin{proof}
We can write
\[T_m(P_{k_1}f,P_{k_2}g)(x)=\int K(x-y,x-y')f(y)g(y')dydy'\]
where the kernel
\[K(x,y)=\int m(\xi,\eta)\chi_{k_1}(\xi)\chi_{k_2}(\eta)e^{ix\xi+iy\eta}d\xi d\eta.\]
From the assumption on $m$, and integration by parts, we get a
pointwise bound
\[|K(x,y)|\lec
2^{k_1n}(1+|2^{k_1}x|)^{-n-1}2^{k_2n}(1+|2^{k_2}y|)^{-n-1}.\] Thus
the lemma follows from H\"older's inequality and Minkowski's inequality.
\end{proof}

\begin{lem} For any $N$ and $u$ we have
\EQ{
 \pt\|\LR{D}\Om(N,u)\|_{L^2_t\dot B^{1/4+\e}_{q(\e)}}
  \lec \|N\|_{L^\I_t L^2_x}\|u\|_{L^2_t H^1_6},
 \pr\|D\tilde\Om(u,u)\|_{L^2_t \dot B^{-1/4-\e}_{q(-\e)}}\lec \|u\|_{L^2_tL^6_x}\|u\|_{L^\I_tL^2_x}.}
\end{lem}
\begin{proof}
For the first inequality, it suffices to prove
\[\|\LR{D}\Om(N,u)\|_{\dot B^{1/4+\e}_{q(\e)}}
  \lec \|N\|_{L^2_x}\|u\|_{H^1_6}.\]
By Sobolev embedding, we get $\|\LR{D}\Om(N,u)\|_{\dot
B^{1/4+\e}_{q(\e)}} \lec \|D\LR{D}\Om(N,u)\|_{L^2}$. It is easy to
see that $D\LR{D}\Om(N,u)$ is a bilinear multiplier with the symbol
\[m(\xi,\eta)=\frac{|\x+\y|\LR{\xi+\eta} \sum_{2^k\nsim
\alpha}\chi_{\leq k-5}(\eta)\chi_k(\xi)}{-|\xi+\eta|^2+\alpha
|\eta|+|\xi|^2},\] and $m$ satisfies the condition in Lemma
\ref{lem:bimult}. Thus applying dyadic decomposition and Lemma
\ref{lem:bimult}, we get
\begin{align*}
\|D\LR{D}\Om(N,u)\|_{L^2}&\lec (\sum_{k_2}\norm{\sum_{k_1\leq k_2-5}D\jb{D}\Om(P_{k_2}N,P_{k_1}u)}_2^2)^{1/2}\\
&\lec (\sum_{k_2}(\sum_{k_1\leq
k_2-5}\norm{P_{k_2}N}_{L^2}\norm{P_{k_1}u}_\I)^2)^{1/2}\lec
\|N\|_{L^2_x}\|u\|_{H^1_6}
\end{align*}
where we used $\sum_{k_1\leq k_2-5}\norm{P_{k_1}u}_\I\lec
\sum_{k_1\leq k_2-5}2^{k_1/2}\norm{P_{k_1}u}_6\leq \|u\|_{H^1_6}$.

Similarly, for the second inequality, by Sobolev embedding we get
\[\|D\tilde\Om(u,u)\|_{\dot B^{-1/4-\e}_{q(-\e)}}
  \lec \|{D}^{3/2}\tilde\Om(u,u)\|_{L^2}\]
and $D^2\tilde\Om$ behaves similarly to $D\jb{D}\Om$. Then applying
dyadic decomposition and Lemma \ref{lem:bimult}, we get
\begin{align*}
\|D^{3/2}\tilde\Om(u,u)\|_{L^2}&\lec (\sum_{k_2}\norm{\sum_{k_1\leq
k_2-5}D^{3/2}\tilde\Om(P_{k_2}u,P_{k_1}u)}_2^2)^{1/2}\\
&\lec (\sum_{k_2}(\sum_{k_1\leq
k_2-5}2^{-k_2/2}\norm{P_{k_2}u}_{L^2}\norm{P_{k_1}u}_\I)^2)^{1/2}\lec
\|u\|_{L^2_x}\|u\|_{L^6_x}.
\end{align*}
Thus we finish the proof of the lemma.
\end{proof}

\subsection{Cubic terms}
Finally, we deal with the cubic terms.
\begin{lem} For any $N$ and $u$ we have
\begin{align*}
\|\Om(D|u|^2,u)\|_{L^1_tH^1_x}&\lec
\|\jb{D}u\|_{L^2_tL^6_x}^2\|u\|_{L^\I_tL^2_x},\\
\|\LR{D}\Om(N,Nu)\|_{L^2_t\dot B^0_{6/5,2}}&\lec
  \|\jb{D}u\|_{L^2_tL^6_x}\|N\|_{L^\I_tL^2_x}^2,\\
 \|D\tilde\Om(Nu,u)\|_{L^1_tL^2_x}&\lec
\|\jb{D}u\|_{L^2_tL^6_x}^2\|N\|_{L^\I_tL^2_x}.
\end{align*}
\end{lem}
\begin{proof}
As in the proof of previous lemma, applying dyadic decomposition and
Lemma \ref{lem:bimult}, we get
\begin{align*}
\|\Om(D|u|^2,u)\|_{H^1_x}&\lec(\sum_{k_2}\norm{\sum_{k_1\leq k_2-5}\jb{D}\Om(P_{k_2}D|u|^2,P_{k_1}u)}_2^2)^{1/2}\\
&\lec (\sum_{k_2}(\sum_{k_1\leq
k_2-5}\norm{P_{k_2}|u|^2}_{L^2}\norm{P_{k_1}u}_{L^\I})^2)^{1/2}\lec
\|\jb{D}u\|_{L^6_x}^2\|u\|_{L^2_x}.
\end{align*}
Similarly, for the second inequality, we have
\begin{align*}
\|\LR{D}\Om(N,Nu)\|_{\dot B^0_{6/5}}&\lec(\sum_{k_2}\norm{\sum_{k_1\leq k_2-5}\jb{D}\Om(P_{k_2}N,P_{k_1}(Nu))}_{6/5}^2)^{1/2}\\
&\lec (\sum_{k_2}(\sum_{k_1\leq
k_2-5}2^{-k_2}\norm{P_{k_2}N}_{L^2}\norm{P_{k_1}(Nu)}_{L^3_x})^2)^{1/2}
 \\&\lec (\sum_{k_2}(\sum_{k_1\leq
k_2-5}2^{-k_2+k_1}\norm{P_{k_2}N}_{L^2}\|N\|_{L^2_x}\|u\|_{L^6_x})^2)^{1/2}
 \\&\lec \|u\|_{L^6_x}\|N\|_{L^2_x}^2,
\end{align*}
and for the last inequality, we have 
\EQ{
 \|D\tilde\Om(Nu,u)\|_{L^2_x} 
 \pt\lec (\sum_{k_2}\norm{\sum_{k_1\leq k_2-5}2^{k_2}\ti\Om(P_{k_2}(Nu),P_{k_1}u)}_2^2)^{1/2}
 \pr\lec (\sum_{k_2}(\sum_{k_1\leq k_2-5}\LR{2^{k_2}}^{-1}\norm{P_{k_2}(Nu)}_{L^2}\norm{P_{k_1}u}_{L^\I})^2)^{1/2}
 \pr\lec \|\jb{D}u\|_{L^6_x}^2\|N\|_{L^2_x}.}
Using H\"older in $t$, we obtain the claimed inequalities. 
\end{proof}

\section{Proof of Theorem \ref{thm}}

Now we are ready to prove Theorem \ref{thm}. The main ingredients
are the estimates obtained in the previous section. For any
$(u_0,n_0,n_1)\in H^1_r(\R^3)\times L^2_r(\R^3)\times \dot
H^{-1}_r(\R^3)$, we define an operator $\Phi_{u_0,n_0,n_1}(u,N)$ by
the right-hand side of \eqref{eq:intu}-\eqref{eq:intN}. Our
resolution space is
\[X_\eta=\{(u,N): \norm{(u,N)}_X=\norm{\jb{D}u}_{L^\I_tL^2_x \cap L^2\dot B^{1/4+\e}_{q(\e),2}}+\norm{N}_{L^\I_tL^2_x \cap L^2_t\dot B^{-1/4-\e}_{q(-\e),2}}\leq \eta\}\]
endowed with the norm metric $\norm{\cdot}_X$.

We will show that $\Phi_{u_0,n_0,n_1}:X_\eta\to X_\eta$ is a
contraction mapping, provided that $\eta\ll 1$ and $(u_0,n_0,n_1)$
are sufficiently small. By the estimates in the previous section, we
have for any $(u,N)\in X_\eta$
\begin{align*}
\norm{\Phi_{u_0,n_0,n_1}(u,N)}_X&\lec
\norm{u_0}_{H^1}+\norm{n_0}_{L^2}+\norm{n_1}_{\dot
H^{-1}}+(\norm{u_0}_{H^1}+\norm{n_0}_{L^2}+\norm{n_1}_{\dot
H^{-1}})^2\\
&+\norm{(u,N)}_X^2+\norm{(u,N)}_X^3\leq \eta
\end{align*}
if $\e_0=\norm{u_0}_{H^1}+\norm{n_0}_{L^2}+\norm{n_1}_{\dot
H^{-1}}\ll 1$, and we set $\eta=C\e_0$. Similarly, we can prove
$\Phi_{u_0,n_0,n_1}:X_\eta\to X_\eta$ is a contraction mapping. Our
estimates are time global, therefore Theorem \ref{thm} follows
immediately. The quadratic terms without time integral tend to $0$ strongly in the energy space as $|t|\to\I$, so they do not contribute to the scattering states. 

\subsection*{Acknowledgment}
This work started when the second author was visiting the Beijing
International Center for Mathematical Research. He would like to
thank Professor Baoxiang Wang for his kind invitation and
hospitality. Z. Guo is supported in part by NNSF of China (No.
11001003).

\end{document}